\documentclass[12pt,a4paper]{amsart}

\usepackage[margin=2cm]{geometry}

\RequirePackage{amsmath}
\RequirePackage{amssymb}
\RequirePackage{amsthm}
\RequirePackage{bm}

\usepackage{tikz}
\usetikzlibrary{matrix,arrows}
\usetikzlibrary{calc}

\newcommand{\Q}{\ensuremath{\mathbb{Q}}}

\newcommand{\Z}{\ensuremath{\mathbb{Z}}}

\newcommand{\eps}{\epsilon}

\newcommand{\abs}[1]{\ensuremath{\lvert#1\rvert}}

\newcommand{\abss}[1]{\ensuremath{\left\lvert#1\right\rvert}}

\DeclareMathOperator{\rank}{rank}

\DeclareMathOperator{\supp}{supp}
\DeclareMathOperator{\codim}{codim}
\DeclareMathOperator{\Sym}{Sym}

\DeclareMathOperator{\Aut}{Aut}
\DeclareMathOperator{\coker}{coker}
\DeclareMathOperator{\Spec}{Spec}

\newtheorem{theorem}{Theorem}[section]
\newtheorem{lemma}[theorem]{Lemma}
\newtheorem{proposition}[theorem]{Proposition}
\newtheorem{corollary}[theorem]{Corollary}

\theoremstyle{definition}

\newlength{\youngscale}
\newlength{\griddiagscale}
\newcommand{\defn}[1]{\textbf{#1}}

\title{Cokernels of random matrices satisfy the Cohen-Lenstra heuristics}
\author{Kenneth Maples}
\address{Institut f\"ur Mathematik, Universit\"at Z\"urich, Winterthurerstrasse 190, CH-8057 Z\"urich}
\email{kenneth.maples@math.uzh.ch}
\date{7/5/13}

\begin{document}
\begin{abstract}
Let $A$ be an $n \times n$ random matrix with iid entries taken from the $p$-adic integers or $\Z/N\Z$. Then under mild non-degeneracy conditions the cokernel of $A$ has a universal probability distribution. In particular, the $p$-part of a random matrix over $\Z$ has cokernel distributed according to the Cohen-Lenstra measure,
\[
  \mathbb{P}(\coker A \cong G) = \frac{1}{\abs{\Aut G}} \prod_{k = 1}^\infty (1 - p^{-k}) + O(e^{-cn})
\]
where the constants only depend on the min-entropy of the entry measure.
\end{abstract}

\thanks{The author was partially supported by a Graduate Research Fellowship from the National Science Foundation. This article is an outgrowth of the author's work on his Ph.D.~thesis \cite{PHD}}

\maketitle

\section{Introduction}

The Cohen-Lenstra measure is a probability distribution on isomorphism classes of finite abelian $p$-groups, given by
\[
  \mu_{\text{CL}}(G) := \frac{1}{\abs{\Aut G}} \prod_{k=1}^\infty (1 - p^{-k}).
\]
It appears in various guises as a candidate for a random abelian $p$-group. Its best known appearance is in the work of Cohen and Lenstra \cite{CL84}, where based on numerical and heuristic evidence they conjectured that the $p$-part of the ideal class group of imaginary quadratic number fields are distributed according to $\nu$. These conjectures were extended to other number fields, in particular by Malle \cite{Mal08}, as long as the number field does not contain $p$th roots of unity.

It was observed in \cite{FW89} that if $B : \Z_p^n \to \Z_p^n$ is a random matrix with iid entries chosen according to Haar measure, then the cokernel distribution of $B$ converges to the Cohen-Lenstra measure as $n \to \infty$. In this article we show that this property holds for random matrices in a much larger class, which suggests that the Cohen-Lenstra heuristics should not strongly depend on the construction of the group.

Let $\xi \in \Z_p$ be a random variable. We define the min-entropy $\alpha = \alpha(\xi)$ to be the largest value $0 < \alpha < 1$ such that we have the non-degeneracy condition
\[
  \sup_{t \in \Z/p\Z} \mathbb{P}(\xi = t \bmod p) \leq 1 - \alpha.
\]
If $A \in M(n,\Z_p)$ is an iid random matrix whose entries have min-entropy $\alpha$, then we say that $A$ has min-entropy $\alpha$ as well.

We have the following universality principle for random matrices with iid entries in the $p$-adic integers.
\begin{theorem} \label{thm:padiccokernel}
Let $A \in M(n,\Z_p)$ be an iid random matrix with min-entropy $\alpha$. Then for all finite abelian $p$-groups $G$,
\[
  \mathbb{P}(\coker A \cong G) = \frac{1}{\abs{\Aut G}} \prod_{k=1}^\infty (1 - p^{-k}) + O(e^{-c\alpha n})
\]
where the constants are absolute.
\end{theorem}
We can also control the probability distribution of the cokernel of matrices over other rings. Let $N$ denote a positive integer with $\omega = \omega(N)$ distinct prime factors. Let $B : (\Z/N\Z)^n \to (\Z/N\Z)^n$ be chosen uniformly from all $\Z/N\Z$-module morphisms and let $A$ be a random $n \times n$ matrix over $\Z/N\Z$ with iid entries distributed according to a random variable $\xi$. We define the min-entropy $\alpha(\xi)$ analogously to the $p$-adic case; namely, $0 < \alpha < 1$ is the largest number such that
\[
  \sup_{p \mid N} \sup_{t \in \Z/p\Z} \mathbb{P}(x = t \bmod p) \leq 1 - \alpha.
\]
Then we have the following universality result.
\begin{theorem} \label{thm:compositecokernel}
Let $A \in M(n,\Z/N\Z)$ be an iid random matrix with min-entropy $\alpha$. Then for all finite $\Z/N\Z$-modules $G$,
\[
  \mathbb{P}(\coker A \cong G) = \mathbb{P}(\coker B \cong G) + O_\omega(e^{-c\alpha n})
\]
where the constants depend only on $\omega = \omega(N)$, the number of distinct prime factors of $N$.
\end{theorem}

As a corollary, we have the following generalization of Theorem~1.2 from \cite{M10a} to control the rank of random matrices over $\Z/p\Z$ for primes $p$.
\begin{corollary}
Let $A \in M(n,\Z/p\Z)$ be a iid random matrix with min-entropy $\alpha$. Then we have
\[
  \mathbb{P}(\rank A = n - k) = p^{-k^2} \frac{\prod_{\ell = k+1}^\infty (1 - p^{-\ell})}{\prod_{\ell = 1}^k (1 - p^{-\ell})} + O(e^{-c\alpha n})
\]
for all $0 \leq k \leq n$, where the constants are absolute.
\end{corollary}

From the above results, we can recover the torsion-free results of Tao and Vu \cite{TV06b}, which control the probability that an iid random matrix over a torsion-free ring is singular. We are also able to recover the estimates over prime fields of Charlap, Rees, and Robbins \cite{CRR90}, Kahn and Kom\l\'os \cite{KK01}, and the author \cite{M10a}. There does not appear to be any previous estimates for the cokernels of non-uniform random matrices in the literature.

It is also easy to see that the above estimates are sharp up to the constants. In fact, let $A$ be a random $0,1$-matrix with iid entries $\xi$ such that $\mathbb{P}(\xi = 1) = \alpha$. then each column is zero with probability $(1-\alpha)^{-n} = O(e^{-c \alpha n})$. As $A$ has trivial cokernel over the $p$-adic integers if and only if it is invertible mod $p$, and the Cohen-Lenstra measure assigns the (positive, independent of $n$) probability $\prod_{\ell \geq 1} (1 - p^{-\ell})$ to the trivial cokernel, we see that the error in Theorem~\ref{thm:padiccokernel} is necessary.

For now let $R = \mathbb{Z}_p$ or $\Z/N\Z$ and let $X_1, ..., X_n$ denote the columns of $A$. The main idea behind Theorem~\ref{thm:padiccokernel}~and~\ref{thm:compositecokernel} is to expose the columns one by one and compute the probability distribution of the successive quotients
\[
  R^n / \langle X_{k+1}, ..., X_n \rangle
\]
where $\langle X_{k+1}, ..., X_n \rangle$ denotes their span as $R$-modules. Let $W_k$ denote the quotient module $R^n / \langle X_{k+1}, \ldots, X_n \rangle$. Then $W_k$ will, with high probability, be isomorphic to $R^k \times T_k$ for some \emph{finite} $R$-module $T_k$. Conditioning on the isomorphism class for $T_j$ with $j > k$, we can partition $R^n$ into sets according to the resulting class of $T_k$. These subsets of $R^n$ can be written as the set-theoretic difference of $R$-submodules of $R^n$ that are constructed in a natural way from $\langle X_{k+1}, \ldots, X_n \rangle$.

It then remains to compute the probability that $X_k$ lies in one of the submodules of $R^n$ from the previous decomposition. This can be done by classifying the submodule according to its combinatorial structure. The analysis of these submodules relies on the swapping argument of Tao and Vu \cite{TV06b, TV07}, based on work of Kahn, Koml\'os, and Szemer\'edi \cite{KKS95} which the author subsequently adapted to finite fields \cite{M10a}. These arguments us to show that enlarged submodules $N$ where $\mathbb{P}(X_k \in N)$ deviates significantly from $\abs{N^\perp}^{-1}$ can be induced with much higher probability by a different probability distribution.

For the remaining submodules, we still have $\mathbb{P}(X_k \in N)$ deviating slightly from uniform. In this setting we generalize the inverse theorem of \cite{M10a} to show that $N^\perp$ contains a ``structured'' vector. Since we can assume that this probability is within a constant factor of $\abs{N^\perp}^{-1}$, it is possible to explicitly enumerate all such vectors.

The proof relies heavily on the independence of the entries of the matrix. It is likely that the proof could be generalized as in the work of Bourgain, Vu, and Matchett-Wood \cite{BVW10} to consider matrices with independent but not necessarily identically distributed entries. It is also likely that some dependencies could be introduced among the columns (or rows) of the matrix, so long as the Fourier-analytic arguments can be preserved. No attempt was made to optimize the constants appearing in Theorems~\ref{thm:padiccokernel}~and~\ref{thm:compositecokernel}.

Although we only work over $R = \Z_p$ and $\Z/N\Z$ in this article, the argument is amenable to much more general analysis. We do not attempt maximum generality to avoid unnecessary complexity for the most interesting cases. It would, however, be nice to extend Theorem~\ref{thm:padiccokernel} to the ring of integers of finite extensions of $\Q_p$. It seems that the arguments in this article suffice if we assume that the probability distribution on the entries, taken modulo the maximal ideal $\mathfrak{m}$, are ``non-degenerate'' in the sense that they do not concentrate on an \emph{additive cosets} of $\mathbb{F}_{p^f} \cong R/\mathfrak{m}$. However, it is intuitively clear that it should suffice for bounded exponents $f$, for the probability distribution to not concentrate on \emph{affine subfields}; i.e.~subsets of the form $\beta \mathbb{F}_{p^d} + \gamma$ for some $d \mid f$.

\section{Notation} \label{sec:notation}

We will use standard asymptotic notation. For an index variable $n$ and functions $f, g$ of $n$, we write $f = O(g)$ to mean that there are absolute constants $n_0$ and $C$ such that for all $n > n_0$, $f(n) \leq C g(n)$. These constants may change from line to line.  Similarly, the equation $f = g + O(h)$ is shorthand for $\abs{f - g} = O(h)$. We use the notations $f \lesssim g$ and $g \gtrsim f$ to denote $f = O(g)$ when convenient.

It is convenient to employ probabilistic notation. If $E$ is an event, we let $\mathbb{P}(E)$ denote its probability; if $F$ is another event, then we can express the conditional probability of $E$ on $F$ as $\mathbb{P}(E \mid F)$ and their intersection as $\mathbb{P}(E \wedge F)$. For random variables $X$ we let $\mathbb{E} X$ denote their expectation.

For $V$ an $R$-submodule of $R^n$, we define
\[
  V[t] := \{v \in R^n \mid tv \in V\}
\]
and by abuse of notation
\[
  V[\infty] := \{v \in R^n \mid tv \in V \text{ for some non-zero } t \in R\}.
\]

For positive integers $m,n$ we will let $[m,n]$ denote their least common multiple.

We will use the number theorist's exponential function $e(t) := \exp(2 \pi i t)$ and $e_p(t) := \exp (2 \pi i t / p)$.

If $A$ is a set we will let $\#A$ and $\abs{A}$ both denote its cardinality. We write $[n] = \{1, \ldots, n\}$ for the indicated set.

\section{The column exposure process} \label{sec:columnexposureprocess}

\subsection{Statement of the universality proposition}

Consider the set of $n \times n$ matrices over the finite field $\mathbb{F}_p$. It is well-known that the proportion of invertible matrices in this set can be calculated from the sequence of column vectors $X_1, \dotsc, X_n$ and the associated sequence of subspaces
\[
  W_\ell := \langle X_{\ell+1}, \dotsc, X_n \rangle \subseteq \mathbb{F}_p^n
\]
for $\ell = 0, \dotsc, n$. Here we have chosen to expose the columns from $X_n$ to $X_1$ so that the dominant contribution to the cokernel will be for small indices; in particular, this greatly simplifies calculation. To compute this proportion, we see that the entire matrix is invertible if and only if $\codim W_\ell = n - \dim W_\ell = \ell$ for each $0 \leq \ell \leq n-1$. If we let $A$ denote a random $n \times n$ matrix over $\mathbb{F}_p$ chosen uniformly, so that the columns $X_1, \dotsc, X_n$ are independent random vectors taken uniformly from $\mathbb{F}_p^n$, then we compute
\begin{align*}
  \mathbb{P}(\codim W_{\ell-1} = \ell-1 \mid \codim W_\ell = \ell) &= 1 - \mathbb{P}(X_\ell \in W_\ell \mid \codim W_\ell = \ell) \\
  &= 1 - p^{-\ell}
\end{align*}
so, in particular,
\[
  \mathbb{P}(A \text{ is invertible}) = \prod_{\ell=1}^n (1 - p^{-\ell}).
\]

We can generalize this argument to compute the cokernel of a random $n \times n$ matrix $A$ over the ring $R$, given by
\[
  \coker A := R^n / \langle X_1, ..., X_n \rangle.
\]
We consider the sequence of submodules
\[
  W_\ell := \langle X_{\ell + 1}, ..., X_n \rangle
\]
as $\ell = 0, \dotsc, n$. We have the natural quotient map
\[
  \phi_\ell : R^n / W_\ell \longrightarrow R^n / W_{\ell - 1}
\]
with kernel equal to the $R$-span of $X_\ell$ in $R_n / W_\ell$.

In this section we show that the isomorphism class of these quotients can be computed by testing the membership of the column vectors $X_\ell$ in various submodules constructed in a natural way from $W_\ell$.

Let $\mathcal{F}$ be a finite set of pairs in $R \times (R \cup \{\infty\})$. For a given submodule $W \subset R^n$, we define the \textbf{enlarged submodule} $\phi_{\mathcal{F}}(W)$ to be
\[
  \phi_{\mathcal{F}}(W) = \bigcap_{(a,b) \in \mathcal{F}} a R^n + W[b]
\]
where
\[
W[t] := \{v \in R^n \mid tv \in W\}
\]
and
\[
W[\infty] := \{v \in R^n \mid tv \in W \text{ for some non-zero } t \in R\}.
\]
Since enlarged submodules are functorial in $W_\ell$, we will denote the enlarged submodule with index set $\mathcal{F}$ by $\phi_{\mathcal{F}}(W_\ell)$.

The following universality result forms the heart of the argument.
\begin{proposition} \label{prop:universality}
For some $0 < \eta < 1$ and all $k \leq \eta n$, the following holds. Let $\xi \in R$ be a random variable with min-entropy $\alpha$. Let $X, X_{k+1}, \ldots, X_n \in R^n$ be iid random vectors with independent coefficients distributed according to $\xi$. Let $W_k = \langle X_{k+1}, \ldots, X_n \rangle$ denote the span of the $n-k$ independent vectors. Then for every $R$ submodule $V \subset R^n$ and finite index set $\mathcal{F}$,
\[
  \mathbb{P}(X \in \phi_{\mathcal{F}}(W_k) \mid R^n / \phi_{\mathcal{F}}(W_k) \cong V^\perp) = \abs{V^\perp}^{-1} + O(e^{-c \alpha n})
\]
where the constants are absolute. In particular, if $V$ is not cofinite, then the indicated probability is bounded by $O(e^{-c \alpha n})$.
\end{proposition}
Note that $\abs{V^\perp}^\perp$ is the probability that a \emph{uniform} random $X$ lies in a submodule $Y$ with $R^n / Y \cong V^\perp$.

The bulk of this article concerns the proof of Proposition~\ref{prop:universality}. The goal of this section is to show how Proposition~\ref{prop:universality} implies Theorems~\ref{thm:padiccokernel}~and~\ref{thm:compositecokernel}. Before we do so, it is convenient to recall that Theorem~\ref{thm:padiccokernel} holds if we can show that the cokernel distribution for $A$ is exponentially close to the cokernel distribution of a uniform random matrix over $\Z_p$, which is shown in e.g.~\cite{FW89} with a lemma from \cite{CL84}.

\begin{proposition}
Let $A \in M(n,\Z_p)$ be taken according to the uniform measure. Then for any finite abelian $p$-group $G$,
\[
  \mathbb{P}(\coker A \cong G) = \frac{1}{\abs{\operatorname{Aut} G}} \prod_{k=1}^n (1 - p^{-k}) \prod_{j=n-r+1}^n (1 - p^{-j})
\]
where $r = \rank_{\mathbb{F}_p} G / p G$.
\end{proposition}

\subsection{The sequence of partial cokernels}

The partial quotients contain information about the syzygies of $X_{\ell + 1}, ..., X_n$ over quotients of $R$. We first observe that for $R = \Z_p$ we can extract the free part of $\Z_p^n / W_\ell$ almost surely.
\begin{proposition} \label{prop:noxintorsion}
For $R = \Z_p$ we have $\mathbb{P}(X_\ell \in W_\ell[\infty]) = O(e^{-c \alpha n})$.
\end{proposition}
\begin{proof}
For each $L \in \Z^+$ set $\mathcal{F}(L) = \{(p^L, \infty)\}$. Then $\phi_{\mathcal{F}(L)}(W_\ell) = p^L R^n + W_\ell[\infty]$ and
\[
  \bigcap_{L=1}^\infty \phi_{\mathcal{F}(L)}(W_\ell) = W_\ell[\infty].
\]
Now $\phi_{\mathcal{F}(L)}(W_\ell)^\perp \cong (\Z/p^L\Z)^j$ for some $j \geq \ell$ because $W_\ell$ is spanned by $n - \ell$ vectors in $\Z_p^n$. Thus by Proposition~\ref{prop:universality} we have
\[
  \mathbb{P}(X_\ell \in \phi_{\mathcal{F}(L)}(W_\ell)) \leq p^{-L} + O(e^{-c \alpha n})
\]
and the result follows from the dominated convergence theorem.
\end{proof}

\begin{proposition} \label{prop:takepparts}
    For $R = \Z/N\Z$ we can factor
    \[
      R^n / W_\ell = \bigoplus_{p \mid N} R_{(p)}^n / (W_\ell)_{(p)}
    \]
    where $R_{(p)}$, $(W_\ell)_{(p)}$ denote the $p$-part of $R$, $W_\ell$ respectively.
\end{proposition}
\begin{proof}
    This is immediate from the Chinese remainder theorem.
\end{proof}
\begin{proposition} \label{prop:splitfreetorsion}
With probability $1 - O(e^{-c \alpha n})$ there are isomorphisms
    \[
      R^n / W_\ell \cong R^\ell \oplus T_\ell
    \]
where $T_\ell$ is a finite $R$-module with unique isomorphism class.
\end{proposition}
\begin{proof}
First suppose $R = \Z_p$. We induct downward on $\ell$ starting with $\ell = n$. For $\ell = n$ the result is trivial with $T_n = 0$. If the result is shown for $\ell$, then it suffices to find a finite $\Z_p$-module $T_{\ell-1}$ such that $\Z_p^{\ell-1} \oplus T_{\ell-1} \cong \Z_p^\ell \oplus T_\ell / \langle X_\ell \rangle$. Let $v_1, ..., v_\ell$ be some basis for $\Z_p^\ell$ and write
\[
  X_\ell = a_1 v_1 + \cdots + a_\ell v_\ell + t
\]
with $t \in T_{\ell}$; this decomposition is unique. By Proposition~\ref{prop:noxintorsion} we know that with probability $1 - O(e^{-c \alpha n})$ at least one of the coefficients $a_1, \dotsc, a_\ell$ is non-zero. Suppose without loss of generality that the power of $p$ dividing $a_{\ell}$ is the smallest among $a_1, \dotsc, a_\ell$. Consider the short exact sequence
\begin{center}
\begin{tikzpicture}[description/.style={fill=white,inner sep=2pt}]
	\matrix (m) [matrix of math nodes, row sep=3em, column sep=2.5em,
				 text height=1.5ex, text depth=0.25ex]
    { 0 & \Z_p^{\ell-1} & \Z_p^\ell \oplus T_\ell / \langle X_\ell \rangle & \Z_p^\ell \oplus T_\ell / \langle X_\ell, v_1, ..., v_{\ell-1} \rangle & 0 \\ };
    \path[->,font=\scriptsize]
        (m-1-1) edge (m-1-2)
        (m-1-2) edge node[auto] {$ \iota $} (m-1-3)
        (m-1-3) edge node[auto] {$ \pi $} (m-1-4)
        (m-1-4) edge (m-1-5);
\end{tikzpicture}
\end{center}
where $\iota$ is the inclusion map on $v_1, ..., v_{\ell-1}$ and $\pi$ is the indicated quotient map. It is easy to see that $\Z_p^\ell \oplus T_\ell / \langle X_\ell, v_1, ..., v_{\ell-1} \rangle$ is finite: each element of $T_\ell$ is torsion, while the generators of $\Z_p^\ell$ either map to $0$ or are torsion. Let
\[
  T_{\ell - 1} := \Z_p^\ell \oplus T_\ell / \langle X_\ell, v_1, ..., v_{\ell-1} \rangle.
\]
By the splitting lemma, it suffices to define a map of $\Z_p$-modules $\psi : \Z_p^\ell \oplus T_\ell / \langle X_\ell \rangle \to \Z_p^{\ell - 1}$ such that $\psi \circ \iota$ is the identity on $\Z_p^{\ell - 1}$.

Indeed, for $w \in \Z_p^\ell \oplus T_\ell / \langle X_\ell \rangle$ choose a representative
\[
  w = b_1 v_1 + \cdots + b_\ell v_\ell + t'
\]
where $t' \in T_\ell$. Now $b_\ell = p^r \beta$ for some $r \geq 0$ and $\beta \in \Z_p^\times$; similarly let $a_\ell = p^s \alpha$ for some $s \geq 0$ and $\alpha \in \Z_p^\times$. If $r \geq s$ then $b_\ell/a_\ell \in \Z_p$ and so there is a unique representative
\[
  w - \frac{b_\ell}{a_\ell} X_\ell \in \langle v_1, \dotsc, v_{\ell-1} \rangle \oplus T_\ell.
\]
We then let $\psi(w)$ be the projection onto $\langle v_1, ..., v_{\ell-1} \rangle$ of this representative. If $r < s$, then $p^{s-r} b_\ell / a_\ell \in \Z_p$ and so there is a unique representative of $p^{s-r}w$ given by
\[
  p^{s-r} w - \frac{p^{s-r} b_\ell}{a_\ell} X_\ell \in \langle v_1, \dotsc, v_{\ell-1} \rangle \oplus T_\ell.
\]
Furthermore, the coefficients of $v_1, \dotsc, v_{\ell-1}$ of this representative must all be divisible by $p^{s-r}$ (this is because $a_\ell$ generates the largest ideal), so we let $\psi(w)$ be $p^{-(s-r)}$ times the projection onto $\langle v_1, ..., v_{\ell-1} \rangle$ of this representative. It is clear that the resulting map is linear and commutes with multiplication by elements of $\Z_p$, as required.

If $R = \Z/N\Z$, by the chinese remainder theorem it suffices to prove the result for $N = p^r$ for some $r \geq 1$. But this follows from the previous case by lifting to representatives in $\Z_p$. \qedhere
\end{proof}

As a consequence we have the following.
\begin{corollary} \label{cor:sumoversequences}
    Let $A$ be as above and let $G$ be a finite abelian $R$-module. For each $1 \leq \ell \leq n$, let $T_\ell$ be constructed as in Proposition~\ref{prop:splitfreetorsion}. Furthermore let $T_n = 0$ and $T_0 = \coker A$. Then we have
    \[
      \mathbb{P}(\coker A \cong G) = \sum_{H_{n-1}, \ldots, H_1} \mathbb{P}(\coker A \cong G \text{ and } T_j \cong H_j \text{ for all } j) + O(e^{-c \alpha n}),
    \]
where the sum is over all sequences of finite $R$-modules.
\end{corollary}    

Next, we will restrict further the sequences $H_{n-1}, \ldots, H_1$ appearing in the previous corollary.

\subsection{Young diagrams}
Finite $R$-modules can be classified by their induced Young diagram. Recall that a \defn{Young diagram} (which, when the meaning is obvious, we will call a \defn{diagram}) is a partition, where to the set $[m] = \{1, \ldots, m\}$ and the partition $j_1 \geq j_2 \geq \cdots \geq j_t > 0$ with $j_1 + \cdots + j_t = m$ we have the corresponding Young diagram with $t$ rows, where in the $k$th row there are $j_k$ boxes. For example, the partition of $[7]$ into $4 \geq 2 \geq 1$ corresponds to the Young diagram
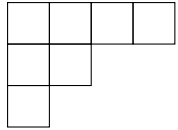
\[
\begin{tikzpicture}[x=\youngscale,y=\youngscale,every node/.style={minimum height=\youngscale, minimum width=\youngscale,draw}]
\draw (0.5,2.5) node {};
\draw (1.5,2.5) node {};
\draw (2.5,2.5) node {};
\draw (3.5,2.5) node {};
\draw (0.5,1.5) node {};
\draw (1.5,1.5) node {};
\draw (0.5,0.5) node {};
\end{tikzpicture}
\]
For $R = \Z_p$, any finite $\Z_p$-module $T$ is a finite abelian $p$-group; by the fundamental theorem of finite abelian groups, we know that $T$ is isomorphic to
\[
  T \cong (\Z/p^{j_1}\Z) \oplus (\Z/p^{j_2}\Z) \oplus \cdots \oplus (\Z/p^{j_t}\Z)
\]
with $j_1 \geq j_2 \geq \cdots \geq j_t > 0$, and this representation (but not the isomorphism!) is unique. We therefore have a one-to-one correspondence between partitions $j_1 \geq j_2 \geq \cdots \geq j_t > 0$ of $[\log_p M]$ and finite $\Z_p$-modules of cardinality $M$.

For $R = \Z/N\Z$, if we factor $R = \oplus_{p^k \| N} \Z/p^k\Z$ then we have a correspondence between finite $R$-modules $T$ and tuples of Young diagrams, one for each prime dividing $N$, such that there are at most $k$ columns in the Young diagram for $p$ if $p^k \| N$. We will call the diagram of $T$ corresponding to the prime $p$ the \defn{$p$-diagram} of $T$.

In the previous part we showed that the matrix $A$ induces a sequence $T_n, \dotsc, T_0$ of finite $R$-modules with probability $1 - O(e^{-c \alpha n})$. It turns out this sequence is strictly growing in the sense that $T_\ell \hookrightarrow T_{\ell - 1}$.
\begin{proposition} \label{prop:onlygrows}
With probability $1 - O(e^{-c \alpha n})$ there is an injection $T_\ell \hookrightarrow T_{\ell - 1}$.
\end{proposition}

\begin{proof}
We see that $\phi_\ell : R^n / W_\ell \to R^n / W_{\ell - 1}$ commutes with the isomorphisms from Proposition~\ref{prop:splitfreetorsion} to give a map
\[
  \phi_\ell : R^\ell \oplus T_\ell \to R^{\ell-1} \oplus T_{\ell - 1}
\]
with kernel equal to $\langle X_\ell \rangle$, the span of the image of $X_\ell$ in $R^\ell \oplus T_\ell$.

If $R = \Z_p$ then by Proposition~\ref{prop:noxintorsion} we have $X_\ell \notin T_\ell$ with probability $1 - O(e^{-c \alpha n})$. In particular $\langle X_\ell \rangle \cap T_\ell = \{0\}$, so $\phi_\ell$ is injective when restricted to $T_\ell$. Since all elements of $T_\ell$ have finite order their images must have finite order as well, and in particular must be contained in $T_{\ell - 1}$. Thus we have $\phi_\ell : T_\ell \hookrightarrow T_{\ell - 1}$ as required.

If $R = \Z/N\Z$, it suffices to prove that $(T_\ell)_{(p)} \hookrightarrow (T_{\ell-1})_{(p)}$. This is obvious after changing coordinates by the classification theorem for finite abelian $p$-groups.
\end{proof}

For any finite $R$-module $T$, we define the $j$-rank of $T$ at $p$ to be
\[
  r_{j,(p)}(T) := \rank_{\mathbb{F}_p}(p^{j-1} T / p^j T).
\]
From the correspondence we see that $r_{j,(p)}$ is equal to the length of the $j$th column of the $p$-diagram of $T$. It turns out that there is considerable rigidity in the change of $r_{j,(p)}(T_\ell)$ as $\ell$ varies.
\begin{proposition} \label{prop:rankchanges}
For all $j \geq 1$, $1 \leq \ell \leq n$, and $p$ we have
\[
  r_{j,(p)}(T_\ell) \leq r_{j,(p)}(T_{\ell - 1}) \leq r_{j,(p)}(T_\ell) + 1.
\]
Furthermore, $r_{j,(p)}(T_{\ell - 1}) = r_{j,(p)}(T_\ell) + 1$ if and only if
\[
\langle X_\ell \rangle \cap p^{j-1} (R^\ell \oplus T_\ell) \bmod p^j = \{0\}
\]
\end{proposition}
\begin{proof}
Let $\phi_\ell : R^\ell \oplus T_\ell \to R^{\ell - 1} \oplus T_{\ell - 1}$ be the map from the proof of Proposition~\ref{prop:onlygrows}. We can restrict $\phi_\ell$ to the submodule $p^{j-1}(R^\ell \oplus T_\ell)$ to derive a map $\phi_\ell' : p^{j-1}(R^\ell \oplus T_\ell) \rightarrow p^{j-1}(R^{\ell - 1} \oplus T_{\ell - 1})$. We can then quotient by $p^j$ and note that $p^j(R^\ell \oplus T_\ell)$ maps to zero to derive a map
\[
  \widetilde{\phi} : p^{j-1}(R^\ell \oplus T_\ell) / p^j (R^\ell \oplus T_\ell) \rightarrow p^{j-1}(R^{\ell-1} \oplus T_{\ell-1}) / p^j (R^{\ell-1} \oplus T_{\ell-1})
\]
This is a map of finite dimensional vector spaces over $\mathbb{F}_p$. The quotients factor over the sums so we can write
\[
\widetilde{\phi} : \mathbb{F}_p^\ell \oplus (p^{j-1} T_\ell / p^j T_\ell) \rightarrow \mathbb{F}_p^{\ell - 1} \oplus (p^{j-1} T_{\ell - 1} / p^j T_{\ell - 1})
\]
$\widetilde{\phi}$ is the composition of two quotient maps so it is surjective. Therefore, by the definition of the $j$-rank, we have
\[
\ell + r_j(T_\ell) \geq \ell - 1 + r_j(T_{\ell - 1}).
\]
The other inequality follows immediately from Proposition~\ref{prop:onlygrows}.

For the second statement, note that $\widetilde{\phi}$ is an isomorphism if and only if the kernel of $\phi$ restricted to $p^{j-1}(R^\ell \oplus T_\ell)$ is contained in $p^j(R^\ell \oplus T_\ell)$.
\end{proof}
This last proposition characterizes sequences $H_n, \dotsc, H_0 = \coker A$ which arise from sequences $T_n, \dotsc, T_0 = \coker A$.

Young diagrams themselves offer no advantages over classical notation for partitions. However, as our sequence of finite modules $T_n, \ldots, T_0$ induces a sequence of nested partitions, we can record this information as a Young tableau. On a Young diagram $\lambda$, we define a \defn{numbering} of $\lambda$ from $[n]$ to be an assignment, for each box in the diagram, of an integer from $[n]$. We then say that $\lambda$ is the \defn{shape} of the numbering. We define a \defn{semi-standard tableau} (or just \textbf{tableau}) of shape $\lambda$ to be a numbering which is weakly decreasing along rows and strictly decreasing along columns; for example,
\[
\begin{tikzpicture}[x=\youngscale,y=\youngscale,every node/.style={minimum height=\youngscale, minimum width=\youngscale,draw}]
\draw (0.5,1) node {4};
\draw (1.5,1) node {3};
\draw (2.5,1) node {3};
\draw (3.5,1) node {1};
\draw (0.5,0) node {2};
\draw (1.5,0) node {2};
\draw (0.5,-1) node {1};
\end{tikzpicture}
\]
is a tableau of shape $\lambda = 4 \geq 2 \geq 1$. We will write $H / \lambda$ to indicate that $H$ is a tableau of shape $\lambda$. Finally, to a tableau $H / \lambda$ we define the \defn{sub-diagram of $H$ above $\ell$} to be the diagram consisting of those boxes of $H$ whose labels $j$ satisfy $j \geq \ell$. For example, the sub-diagram of the previous tableau above $2$ is
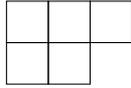
\[
\begin{tikzpicture}[x=\youngscale,y=\youngscale,every node/.style={minimum height=\youngscale, minimum width=\youngscale,draw}]
\draw (0.5,1) node {};
\draw (1.5,1) node {};
\draw (2.5,1) node {};
\draw (0.5,0) node {};
\draw (1.5,0) node {};
\end{tikzpicture}
\]
We will denote the sub-diagram of a tableau $H$ at $\ell$ by $H_\ell$.

As promised above, we can record the sequence of isomorphism classes in a tableau.
\begin{corollary} \label{cor:padictableau}
    Let $A$ be an $n \times n$ matrix over $\Z_p$ and let $\lambda$ denote the diagram of $\coker A$. Then there is a one-to-one correspondence between sequences $T_n, \dotsc, T_0$ with $T_n = \{0\}$ and $T_0 = \coker A$ which can occur as the partial cokernels of $A$ and tableaux $H/\lambda$ from $[n]$, such that $T_\ell$ has diagram equal to $H_\ell$.
\end{corollary}
\begin{proof}
Fix the sequence $T_n, \dotsc, T_0$ arising from a matrix $A$. By Proposition~\ref{prop:onlygrows} we see that the diagrams $\lambda_\ell$ associated to $T_\ell$ are nested, so define $H/\lambda$ to be the numbering of $\lambda$ given by the time the box appears in the nested sequence. Clearly the rows of $\lambda$ are weakly increasing. Proposition~\ref{prop:rankchanges} shows that the columns increase by at most one each step, so there cannot be two adjacent equal numbers in a vertical configuration. The converse is obvious.
\end{proof}

For $R = \Z/N\Z$, we need to construct a tableau for each prime $p$ dividing $N$. Then we have the following.
\begin{corollary} \label{cor:compositetableau}
    Let $A$ be an $n \times n$ matrix over $\Z/N\Z$ and let $\lambda_{(p)}$ denote the diagram for the $p$-part of $\coker A$ for $p \mid N$. Then there is a one-to-one correspondence between sequences $(T_n)_{(p)}, \ldots, (T_0)_{(p)}$ which can occur as the $p$-part of the partial cokernels of $A$ and tableaux $H / \lambda_{(p)}$ from $[n]$ as above.
\end{corollary}
In particular, we need to control a separate tableau for each prime factor of $N$.

\subsection{Enlarged submodules}

From Corollary~\ref{cor:sumoversequences} and Corollary~\ref{cor:padictableau}, for $R = \Z_p$ we can expand
\[
  \mathbb{P}(\coker A \cong G) = \sum_{H / \lambda \text{ from } [n]} \mathbb{P}(T_\ell \cong H_\ell \text{ for all } \ell) + O(e^{-c \alpha n}).
\]
Each term in the sum can be expanded into a product by conditional expectation; namely,
\[
  \mathbb{P}(T_\ell \cong H_\ell \text{ for all } \ell) = \prod_{\ell=0}^n \mathbb{P}(T_\ell \cong H_\ell \mid T_j \cong H_j \text{ for all } j > \ell).
\]
Each term in the product is not yet in a form that we can estimate. In particular, we would like to represent the set of $X_\ell$ such that $T_{\ell-1}$ has the desired isomorphism class as the set-theoretic difference of enlarged submodules.

Consider the two-parameter family of events
\[
  G_{a,b} := \{X_\ell \in p^a R^n + W_\ell[p^b]\}
\]
with $a \geq 1$ and $b \geq 0$. We abuse notation and write
\[
  G_{a,\infty} := \{X_\ell \in p^a R^n + W_\ell[\infty]\}
\]
as well. We have the obvious inclusions
\[
  G_{a+1,b} \subseteq G_{a,b} \subseteq G_{a,b+1}
\]
so we can define events
\[
  E_{j,0} := G_{j,0}
\]
and
\[
  E_{j,t} := G_{j-t,t} \setminus G_{j-t,t-1}.
\]
for $1 \leq j$ and $1 \leq t \leq j-1$. The events $E_{j,t}$ as $t$ varies can be used to test the $j$-rank of the $p$-primary part of $T$ as follows.
\begin{proposition} \label{prop:rgrowsmeansE}
The events $E_{j,0}, \dotsc, E_{j,j-1}$ are disjoint, and the equality
\[
  r_{j,(p)}(T_{\ell - 1}) = r_{j,(p)}(T_\ell) + 1
\]
holds if and only if $E_{j,0} \cup \dotsb \cup E_{j,j-1}$ holds.
\end{proposition}
\begin{proof}
    We have $E_{j,t} \subseteq G_{j-t,t} \subseteq G_{g-t-s,t+s-1}$ for all $s \geq 1$, but $E_{j,t+s} = G_{g-t-s,t+s} \setminus G_{g-t-s,t+s-1}$, so the family is pairwise disjoint.

First we observe that the condition $r_{j,(p)}(T_{\ell - 1}) = r_{j,(p)}(T_\ell) + 1$ is equivalent to $\langle X_\ell \rangle \cap p^{j-1} (R^\ell \oplus T_\ell) \bmod p^j = \{0\}$ by Proposition~\ref{prop:rankchanges}, so it suffices to show that $E_{j,0} \cup \dotsb \cup E_{j,j-1}$ is equivalent to this event.

Suppose that $\langle X_\ell \rangle \cap p^{j-1} (R^\ell \oplus T_\ell) \bmod p^j = \{0\}$. Let $t \geq 0$ be the minimum exponent such that $p^t X_\ell \in p^j R^n + W_\ell$, or equivalently, $X_\ell \in p^{j-t} + W_\ell[p^t]$. If $t = 0$ then $E_{j,0}$ holds; otherwise, $t > 0$ and thus $p^{t-1} X_\ell \notin p^j R^n + W_\ell$. But this implies that $p^{t-1} X_\ell \notin p^{j-1} R^n + W_\ell$, or equivalently $X_\ell \notin p^{j-t} R^n + W_\ell[p^{t-1}]$. Therefore $E_{j,t} = G_{j-t,t} \setminus G_{j-t,t-1}$ holds. It remains to show that $t < j$, but if $t = j$ we would have shown that $p^{t-1} X_\ell \notin p^{j-1} R^n + W_\ell$ which is absurd.

Conversely, suppose $E_{j,t}$ holds. If $t = 0$ then $X_\ell \in p^j R^n + W_\ell$ so $\langle X_\ell \rangle = \{0\} \bmod p^j R^n + W_\ell$. If $t > 0$ then $p^{t-1} X_\ell \notin p^{j-1} R^n + W_\ell$ which implies that $\langle X_\ell \rangle \cap p^{j-1} R^n + W_\ell \subseteq \langle p^t X_\ell \rangle$. However, $p^t X_\ell \in p^j + W_\ell$ so $\langle X_\ell \rangle \cap p^{j-1} R^n + W_\ell = \{0\} \bmod p^j$.
\end{proof}
We have therefore reduced the problem of computing the isomorphism class of $T_{\ell - 1}$ conditioned on $W_\ell$ to testing the membership of
\[
  X_\ell \in p^a R^n + W_\ell[p^b]
\]
for various $a$, $b$, and $p$.

Precisely, we have the following. Fix $\ell$ and let $S \subset \Z^+$ denote those indices where $r_k(H_{\ell-1}) = r_k(H_\ell) + 1$. By Proposition~\ref{prop:noxintorsion}, $S$ is finite with probability $1 - O(e^{-c \alpha n})$. Let $\eta : [\abs{S}] \to \Z^+ \cup \{0\}$ be the function which enumerates the elements of $S$. For example, if $S = \{2, 5, 7\}$ then $\eta(1) = 2$, $\eta(2) = 5$, and $\eta(3) = 7$.

\begin{proposition} \label{prop:congintoenlarged}
    Fix $\ell \geq 1$ and tableau $H / \lambda$. Let $S \subset \Z^+$ and $\eta$ be as above. Then
    \[
      \{T_\ell \cong H_\ell \} = (\neg G_{\abs{S}+1,\infty}) \cap \bigcap_{k =1}^{\abs{S}} E_{\eta(k),\eta(k)-k}.
    \]
\end{proposition}

\begin{proof}
    First suppose that $S$ is empty. Then $T_\ell$ is isomorphic to $H_\ell \cong H_{\ell+1}$ if and only if $X_\ell \notin pR^n + W_\ell[\infty]$, i.e.~$G_{1,\infty}$ does not hold.

Now for non-empty $S$, we must have $E_{\eta(1),t}$ for some $0 \leq t < \eta(1)$ but no $E_{k,s}$ for $k < \eta(1)$ and $0 \leq s < k$. Then the only possibility is $t = \eta(1) - 1$. In fact, for any other $t < \eta(1) - 1$
\[
  E_{\eta(1),t} \subseteq G_{\eta(1)-t,t} \subseteq G_{1,t} \subseteq \bigcup_{k \leq t} E_{k,k-1}.
\]
Now we consider $\eta(2)$. We cannot have $E_{\eta(2),\eta(2)-1}$ because it is disjoint from $E_{\eta(1),\eta(1)-1}$; similarly, we cannot have $E_{\eta(2),t}$ for $t \leq \eta(2) - 3$ because
\[
  E_{\eta(2),t} \subseteq G_{\eta(2)-t,t} \subseteq \bigcup_{0 \leq k \leq \eta(2)-3} E_{k+2,k}
\]
and none of the sets in the union can hold. Thus the only possibility is $E_{\eta(2), \eta(2)-2}$.

If we continue the argument in this way, then it is easy to see that
\[
  \{T_{\ell} \cong H_\ell \} \subseteq \bigcap_{k=1}^{\abs{S}} E_{\eta(k),\eta(k)-k}
\]
and in fact that the right hand side denotes the event that $T_\ell$ matches $H_\ell$ up to the largest element of $S$.

To complete the argument, we must require that no $E_{k+t,t}$ holds for $k > \abs{S}$. But the union of these sets is precisely $G_{\abs{S}+1,\infty}$.
\end{proof}

\subsection{Entropy bounds} \label{sec:entropybounds}

For $\ell > \delta n$ we observe that any non-trivial column would require that $X_\ell \in (p) + W[\infty]$ for some prime $p$. However, by Odlyzko's lemma \cite{Odl88}, this occurs with probability $O(e^{-c \alpha n})$ so these columns can be ignored. We recall the proof here for a finite field, which suffices for $R = \Z_p$ and $\Z/N\Z$.
\begin{lemma}[Odlyzko] \label{lem:odlyzko}
For any fixed subspace $V$ of $\mathbb{F}_q^n$ and random vector $X \in \mathbb{F}_q^n$ with min-entropy $\alpha$, we have the bound
\[
  \mathbb{P}(X \in V) \leq (1-\alpha)^{\codim V}.
\]
\end{lemma}
\begin{proof}[Proof of Lemma~\ref{lem:odlyzko}]
Let $k$ denote the codimension of $V$.  We can find $n - k$ coordinates $\tau \subseteq [n]$ such that $V$ is a graph over $\tau$.  If we condition on the coordinates of $X$ in $\tau$, then there is a unique choice for the remaining coordinates $[n] \setminus \tau$ for $X \in V$.  Since $\mu$ has min-entropy $\alpha$, the probability that each entry of $X$ assumes the required value is bounded by $1 - \alpha$, and the result follows from the independence of the entries.
\end{proof}

For $G$ sufficiently small, i.e. $\log_p \abs{G} = O(1)$ for all $p$, Theorem~\ref{thm:padiccokernel} and Theorem~\ref{thm:compositecokernel} will follow easily from Proposition~\ref{prop:universality} and the decomposition into Young tableaux. However, for larger $G$ there are a super-exponential number of tableaux $H / \lambda$ and we require a finer analysis.

Define the \textbf{weight} $w(\lambda)$ of a diagram $\lambda$ to be the sum
\[
  w(\lambda) = \sum_{(i,j) \in \lambda} \frac{j(j+1)}{2}
\]
Note that $w(\lambda)$ is the sum of the entries of the minimum semi-standard tableau on $\lambda$. Then there are two possibilities.
\begin{proposition} \label{prop:largeweightdiagram}
Let $\eps > 0$ be fixed. Then if $G$ is an $\Z_p$-module such that $w(\lambda) > \eps n$, then Theorem~\ref{thm:padiccokernel} holds.
\end{proposition}

\begin{proof}
We have by Corollary~\ref{cor:sumoversequences} and Corollary~\ref{cor:padictableau} that
\begin{align*}
    \mathbb{P}(\coker A \cong G) &= \sum_{H / \lambda} \mathbb{P}(T_\ell \cong H_\ell \text{ for all } \ell) + O(e^{-c \alpha n}) \\
    &\leq \prod_{\mu} \mathbb{P}(\abs{T_\ell} = p^{\mu(\ell)} \abs{T_{\ell+1}} \text{ for all } \ell) + O(e^{-c \alpha n})
\end{align*}
where the sum is over functions $\mu : [n] \to \Z^+ \cup \{0\}$ with $\sum_t \mu(t) = \abs{w}$. We have the inclusion
\[
  \{\abs{T_\ell} = p^{\mu(\ell)} \abs{T_{\ell+1}}\} \subseteq \{X_\ell \in p^{\mu(\ell)} \Z_p^n + W_\ell[\infty]\} = \{X_\ell \in \phi_{\mathcal{F}(\mu(\ell))}(W_\ell)\}
\]
where $\mathcal{F}(\mu(\ell)) = \{(\mu(\ell),\infty)\}$. By Proposition~\ref{prop:universality}, we have for $\ell < \eta n$ (where $1 > \eta > 0$ is the absolute constant from the proposition) the bound
\[
  \mathbb{P}(X \in \phi_{\mathcal{F}(\mu(\ell))}(W_\ell) \mid R^n / \phi_{\mathcal{F}(\mu(\ell))}(W_\ell) \cong V^\perp) = \abs{V^\perp}^{-1} + O(e^{-c \alpha n})
\]
which, summing over classes of submodules of a given cardinality, gives the bound
\[
  \mathbb{P}(\abs{T_\ell} = p^{\mu(\ell)} \abs{T_{\ell+1}} \mid \abs{T_j} = p^{\mu(j)} \abs{T_{j+1}} \text{ for all } j > \ell) = p^{-\ell \mu(\ell)} + O(e^{-c \alpha n}).
\]
Therefore, collecting terms and using Lemma~\ref{lem:odlyzko} gives
\[
  \mathbb{P}(\coker A \cong G) \leq \sum_{\mu} \prod_{\substack{k < \eta n \\ \mu(k) > 0}} (p^{-k \mu(k)} + e^{-c \alpha n}) \prod_{\substack{k \geq \eta n \\ \mu(k) > 0}} e^{-c \alpha \eta n}.
\]
The number of possible functions $\mu$ is bounded by $\binom{n + \abs{\lambda} + 1}{n} \lesssim e^{c \eps n}$, so it suffices to show that each term in the sum is bounded by $e^{-c \eps n}$. For $\mu$ with $\mu(k) > 0$ for some $k > \eta n$ this is immediate; as is for $\mu$ with $p^{-k \mu(k)} > n^{-1} e^{-c \alpha n}$. Otherwise, we can bound the product by
\[
  \prod_{\mu(k) > 0} p^{-k \mu(k)}(1 + O(n^{-1})) \lesssim p^{- \sum_k k \mu(k)} \lesssim e^{-\eps n}
\]
and the result follows.
\end{proof}

\begin{corollary}
    There is an $\eps > 0$ such that if $G$ is a $\Z/N\Z$ modules such that $w(\lambda_{(p)}) > \eps n$ for some $p \mid N$, then Theorem~\ref{thm:compositecokernel} holds.
\end{corollary}
\begin{proof}
    We have
    \[
      \mathbb{P}(\coker A \cong G) \leq \mathbb{P}(\coker A_{(p)} \cong G_{(p)}) = O(e^{-c \eps \alpha n})
    \]
    where the upper bound follows from Proposition~\ref{prop:largeweightdiagram}.
\end{proof}

If $w(\lambda) < \eps n$, then we can derive a bound on the number of possible diagrams of shape $\lambda$.
\begin{proposition} \label{prop:numberdiagrams}
If $w(\lambda) < \eps n$ then there are $O(e^{-c \sqrt{\eps} n})$ semi-standard tableaux with letters from $[n]$ on $\lambda$.
\end{proposition}
\begin{proof}
Let $d_\lambda(n)$ denote the number of semi-standard tableaux with letters from $[n]$ on $\lambda$. Recall from \cite{Ful97} that
\[
  d_\lambda(n) = \prod_{(i,j) \in \lambda} \frac{n + i - j}{r_i(\lambda) + c_j(\lambda) - j - i + 1}
\]
where $r_i(\lambda)$ is the length of the $i$th row of $\lambda$ and $c_j(\lambda)$ is the length of the $j$th column. We can bound
\begin{align*}
  \prod_{(i,j) \in \lambda} (r_i(\lambda) + c_j(\lambda) - j - i + 1) &\geq \prod_{j=1}^{r_1(\lambda)} (c_j(\lambda)!) \\
  &= \exp(\sum_{j=1}^{r_1} c_j (\log c_j - 1)) O(\exp(c \sqrt{n})).
\end{align*}
Similarly,
\[
  \prod_{(i,j)} (n + i - j) \lesssim \exp (\sum_{j=1}^{r_1} c_j \log n).
\]
Therefore
\[
  d_\lambda(n) \lesssim \exp (\sum_{j=1}^{r_1} c_j (\log n - \log c_j + 1)).
\]
Now
\begin{align*}
    \sum_{j=1}^{r_1} c_j (\log n - \log c_j + 1) &\lesssim \sum_{t = 0}^{\log n} \sum_{j : e^t \leq c_j < e^{t+1}} e^t (\log n - t + 1) \\
    &\lesssim \sum_{t = 0}^{\log n} e^t(\log n - t + 1) \#\{j \mid e^t \leq c_j < e^{t+1}\}.
\end{align*}
Note that
\[
  e^t \#\{j \mid e^t \leq c_j \leq e^{t+1}\}^2 \leq \sum_{i=1}^{c_1} r_i^2 \lesssim \eps n
\]
so that
\[
  \#\{j \mid e^t \leq c_j \leq e^{t+1}\} \leq e^{-t/2} \sqrt{\eps n}
\]
and thus
\[
  d_\lambda(n) \lesssim \exp(\sum_{t=0}^{\log n} e^{t/2} (\log n - t + 1) \sqrt{\eps n}) \lesssim \exp(c n \sqrt{\eps})
\]
as required.
\end{proof}

We are now ready to prove Theorem~\ref{thm:padiccokernel} and Theorem~\ref{thm:compositecokernel}, conditional on Proposition~\ref{prop:universality}.
\begin{proof}[Proof of Theorem~\ref{thm:padiccokernel}]
Let $\eps > 0$ be chosen later and let $\lambda$ be the diagram for $G$. If $w(\lambda) > \eps n$ then we are done by Proposition~\ref{prop:largeweightdiagram}. Otherwise, from Corollary~\ref{cor:sumoversequences} and Corollary~\ref{cor:padictableau} we expand
    \[
      \mathbb{P}(\coker A \cong G) = \sum_{H / \lambda \text{ from } [n]} \mathbb{P}(T_\ell \cong H_\ell \text{ for all } \ell) + O(e^{-c \alpha n}).
    \]
    Expanding the event $\{T_\ell \cong H_\ell \text{ for all } \ell\}$ by conditional expectation, we get
    \[
      \mathbb{P}(\coker A \cong G) = \sum_{H / \lambda \text{ from } [n]} \prod_{\ell=0}^n \mathbb{P}(T_\ell \cong H_\ell \mid T_j \cong H_j \text{ for all } j > \ell) + O(e^{-c \alpha n}).
    \]
    We will abbreviate $C_\ell := \{T_j \cong H_j \text{ for all } j > \ell\}$.
    Let us also fix $H / \lambda$ and $\ell$ for now. Let $S$ denote the set of $j$ such that $r_j(H_\ell) = r_j(H_{\ell+1}) + 1$. By Proposition~\ref{prop:congintoenlarged}, we know
    \[
      \mathbb{P}(T_\ell \cong H_\ell \mid C_\ell) = \mathbb{P}((\neg G_{\abs{S}+1,\infty}) \cap \bigcap_{k =1}^{\abs{S}} E_{\eta(k),\eta(k)-k} \mid C_\ell)
    \]
    where $\eta$ is as defined in that proposition. Each $E_{\eta(k), \eta(k)-k} = G_{k, \eta(k)-k} \setminus G_{k, \eta(k)-k-1}$ (unless $\eta(k) = k$, in which case it is equal to $G_{k,0}$). Let us extend $\eta(\abs{S}+1) = \infty$ for notational convenience. By inclusion-exclusion we have
    \[
      \mathbb{P}(T_\ell \cong H_\ell \mid C_\ell) = \sum_{\sigma \subseteq [\abs{S}+1]} (-1)^{\abs{\sigma}+1} \mathbb{P} (\bigcap_{k \in \sigma} G_{k, \eta(k)-k}) \cap \bigcap_{k \notin \sigma} G_{k, \eta(k)-k-1} \mid C_\ell)
    \]
    For a given $\sigma$, let $\mathcal{F}(\sigma)$ denote the set of pairs of indices appearing in the subscripts of $G$ for the summand. Then we have by definition
    \[
      \mathbb{P}(T_\ell \cong H_\ell \mid C_\ell) = \sum_{\sigma \subseteq [\abs{S}+1]} (-1)^{\abs{\sigma}+1} \mathbb{P}(X_\ell \in \phi_{\mathcal{F}(\sigma)} \mid C_\ell).
    \]
    Because $w(\lambda) < \eps n$ there are at most $2^{c \eps n}$ terms appearing in the sum.
    Let $Y \in \Z_p^n$ be uniformly distributed. For $\ell < \eta n$, by Proposition~\ref{prop:universality} we have
    \[
      \mathbb{P}(X_\ell \in \phi_{\mathcal{F}(\sigma)} \mid C_\ell) = \mathbb{P}(Y \in \phi_{\mathcal{F}(\sigma)} \mid C_\ell) + O(e^{-c \alpha n}).
    \]
    Grouping terms and applying trivial bounds, for $\eps > 0$ sufficiently small we see that 
    \[
      \mathbb{P}(T_\ell \cong H_\ell \mid C_\ell) = \text{uniform} + O(e^{-c \alpha n}).
    \]
    By Proposition~\ref{prop:numberdiagrams} there are at most $O(e^{c \sqrt{\eps} n})$ tableau with letters from $n$. Therefore, rearranging we see that
    \[
      \mathbb{P}(\coker A \cong G) = \text{uniform} + O(e^{-c \alpha n}).
    \]
    Finally, by Proposition~1 of \cite{FW89} we know that
    \[
      \text{uniform} = \mu_{\text{CL}}(G) + O(e^{-cn}) = \frac{1}{\abs{\operatorname{Aut} G}} \prod_{j=1}^n (1 - p^{-j}) + O(e^{-cn})
    \]
    as was to be shown.
\end{proof}

\begin{proof}[Proof of Theorem~\ref{thm:compositecokernel}]
Let $\eps > 0$ be chosen later and let $p_1, \ldots, p_\omega$ be the set of primes dividing $N$. Let $\lambda_{p}$ denote the diagram for $G_{(p)}$. By Corollary~\ref{cor:sumoversequences} and Corollary~\ref{cor:compositetableau} we have
 \begin{align*}
   \mathbb{P}(\coker A \cong G) &= \mathbb{P}(\coker A_{(p)} \cong G_{(p)} \text{ for all } p \mid N) \\
   &= \sum_{\substack{H_{(p)} / \lambda_p \\ p \mid N}} \mathbb{P}(T_{(p),\ell} \cong H_{(p),\ell} \text{ for all } \ell, p) + O(e^{-c \alpha n}).
 \end{align*}
 If any $G_{(p)}$ is a $\Z_p$ module with $w(\lambda) > \eps n$, then the result follows from Proposition~\ref{prop:largeweightdiagram}. Otherwise, we expand by conditional expectation,
\[
  \mathbb{P}(\coker A \cong G) = \sum_{\substack{H_{(p)} / \lambda_p \\ p \mid N}} \prod_{\ell=0}^n \mathbb{P}(T_{(p),\ell} \cong H_{(p),\ell} \forall p \mid T_{(p),j} \cong H_{(p),j} \forall p, j > \ell) + O(e^{-c \alpha n}).
\]
Let
\[
  C_\ell = \{ T_{(p),j} \cong H_{(p),j} \forall p, j > \ell \}.
\]
By Proposition~\ref{prop:congintoenlarged} for each prime we have
\[
  \{T_{(p),\ell} \cong H_{(p),\ell} \text{ for all } p \mid N \mid C_\ell \} = \bigcap_{p \mid N} (\neg G^{(p)}_{\abs{S}+1,\infty}) \cap \bigcap_{k =1}^{\abs{S}} E^{(p)}_{\eta(k),\eta(k)-k}
\]
where $G^{(p)}$ and $E^{(p)}$ are given by
\[
  G_{a,b}^{(p)} = \{X_\ell \in p^a R^n + W_\ell[p^b]\} \qquad G_{a,\infty}^{(p)} = \{X_\ell \in p^a R^n + W_\ell[\infty]\}
\]
and
\[
  E_{j,t}^{(p)} = G_{j-t,t}^{(p)} \setminus G_{j-t,t-1}^{(p)}
\]
as before. We can expand each term in the product by inclusion-exclusion as in the proof of Theorem~\ref{thm:padiccokernel}. There are $O_\omega(2^{c \eps n})$ terms appearing in the sum. For each $\mathcal{F}$ in the expansion, by Proposition~\ref{prop:universality} we have
\[
  \mathbb{P}(X_\ell \in \phi_{\mathcal{F}} \mid C_\ell) = \mathbb{P}(Y \in \phi_{\mathcal{F}}(W_\ell) \mid C_\ell) + O(e^{-c \alpha n})
\]
where $Y$ is chosen uniformly in $(\Z/N\Z)^n$. Grouping terms, we derive
\[
  \mathbb{P}T_{(p),\ell} \cong H_{(p),\ell} \forall p \mid C_\ell) = \text{uniform} + O_\omega(e^{-c \alpha n}).
\]
As before, there are at most $O(e^{c \sqrt{\eps} n})$ tableau with letters from $n$ for each prime. Thus we deduce
\[
  \mathbb{P}(\coker A \cong G) = \text{uniform} + O_\omega(e^{-c \alpha n})
\]
as required.
\end{proof}

\section{Universality for enlarged submodules}

\subsection{Types of enlarged submodules}

We distinguish between four possible types of submodules $N$ that $\phi(W_{\ell + 1}) = \phi_{\mathcal{F}}(W_{\ell+1})$ can represent. Let $\delta$, $d$, and $D$ be constants to be chosen later. Intuitively, $\delta$ and $d$ will be of order $1/100$ while $D$ will be about $10$.
\begin{description}
    \item[sparse] There is a non-zero $w \perp N$ such that $\abs{\supp w} \leq \delta n$.  We will show that sparse submodules are represented with exponentially small probability by direct counting in Section~\ref{sec:sparse}.
    \item[unsaturated] $N$ is not sparse and
    \[
      \max(e^{-d\alpha n}, \frac{D}{\abs{N^\perp}}) \leq \abss{\mathbb{P}(X \in N) - \frac{1}{\abs{N^\perp}}}.
    \]
    We will use a generalized version of the swapping argument from \cite{TV06b}, \cite{TV07} to show that these appear with exponentially small probability in Section~\ref{sec:unsat}.
    \item[semi-saturated] $N$ neither sparse nor unsaturated, and we have the inequality
    \[
      e^{-d\alpha n} \leq \abss{\mathbb{P}(X \in N) - \frac{1}{\abs{N^\perp}}} < \frac{D}{\abs{N^\perp}}
    \]
Note that the set of semi-saturated $N$ may be empty if $\abs{N^\perp}$ is sufficiently large. We will count the semi-saturated submodules by finding a structured $w \perp N$ and then directly counting to show that they are represented with exponentially small probability. This is done in Section~\ref{sec:semisat}.
    \item[saturated] $N$ is neither sparse, unsaturated, nor semi-saturated. In particular, these submodules satisfy the universality property.
\end{description}

Proposition~\ref{prop:universality} follows if we can show that $\phi(W_{\ell})$ represents a saturated subspace with probability $1 - O(e^{-c \alpha n})$, where the constants are absolute.

\subsection{Analytic and Combinatorial Tools}

At this point we must recall some theory from analysis and additive combinatorics. The proofs are omitted; for more information see \cite{TV06a} and \cite{M10a}.

Let $\mu$ be a measure on $R$, which is a compact topological ring. For any $0 < \eps < 1$ we define the spectrum $\Spec_{1-\eps} \mu$ to be the set of Fourier coefficients with magnitude at least $1 - \eps$; i.e.~
\[
  \Spec_{1-\eps} \mu := \{\psi \in \widehat{R} \mid \abs{\mu(\psi)} \geq 1 - \eps\}
\]
The importance of $\Spec$ is that it is closed under a bounded number of set additions, as long as $\eps$ is enlarged sufficiently. Recall that for sets $A, B$ in an additive group $Z$, we define $A + B$ to be the sumset $\{a + b \mid a \in A, b \in B\}$.
\begin{lemma} \label{lem:specadds}
For all $\eps > 0$ and $k$ a positive integer, we have
\[
  \Spec_{1 - \eps} \mu + \cdots + \Spec_{1-\eps} \mu \subseteq \Spec_{1-k^2 \eps} \mu
\]
where there are $k$ summands on the left.
\end{lemma}

Recall that $\Sym(A)$ denotes the largest subgroup of $Z$ such that $A$ is the union of cosets of $\Sym(A)$.
\begin{lemma}[Kneser]
    If $A, B \subseteq Z$ are additive sets in an ambient, finite abelian group $Z$, then we have the bound
    \[
      \abs{A + B} + \abs{\Sym(A + B)} \geq \abs{A} + \abs{B}
    \]
\end{lemma}

Iterating Kneser's inequality we have the following corollary.
\begin{corollary}
Let $A_1, ..., A_k$ and $B$ be additive sets in an ambient finite abelian group $Z$. Suppose we have the sumset inclusion
\[
  A_1 + \cdots + A_k \subseteq B.
\]
Suppose further that $B$ contains no additive cosets of $Z$. Then we have
\[
  \abs{B} + (k-1) \geq \abs{A_1} + \cdots + \abs{A_k}
\]
\end{corollary}

\subsection{Sparse submodules} \label{sec:sparse}

We would like to control the probability that $\phi(W_{\ell + 1})$ is sparse. By definition,
\[
  \mathbb{P}(\phi(W_{\ell+1}) \text{ is sparse}) = \mathbb{P}(\phi(W_{\ell+1}) \perp w \text{ non-zero with } \abs{\supp w} \leq \delta n)
\]
In particular, if we construct the $n \times (n - \ell)$ rectangular matrix
\[
  B := \begin{bmatrix} X_{\ell + 1} & \cdots & X_n \end{bmatrix}
\]
then $w^t B = 0$ for some non-zero $w$ with $\abs{\supp w} \leq \delta n$.

The coefficients of $w$ are contained in a finite $R$-submodule of $\widehat{R}$, since elements of $\widehat{R}$ have finite order. We see that there is a maximal ideal $\mathfrak{m}$ such that $w \bmod \mathfrak{m}$ is not constant zero. Reducing modulo this ideal, and possibly shrinking $\supp w$, we conclude that $w^t B = 0$ modulo $\mathfrak{m}$.

We can now finish the argument as in \cite{M10a}. In particular, we can restrict $B$ to those rows corresponding to non-zero entries of $w$. This gives a $\abs{\sigma} \times n - \ell$ matrix that is not of full rank. Regardless of the choice of spanning columns, the remaining columns must lie in their span and in particular be perpendicular to $w$. We recall the ``classical'' Littlewood-Offord theorem for finite fields from \cite{M10a}.

\begin{lemma} \label{lem:littlewoodofford}
    Let $X \in \mathbb{F}_q^n$ be a random vector with iid entries taken from a probability distribution $\mu$ with min-entropy $\alpha$. Suppose $w \in \mathbb{F}_q^n$ has at least $m$ non-zero coefficients.  Then we have the estimate
    \[
      \abss{\mathbb{P}(X \cdot w = r) - q^{-1}} \lesssim \frac{1}{\sqrt{\alpha m}}
    \]
    for all $r \in \mathbb{F}_q$.
\end{lemma}
Combining these estimates, we see that
\[
  \mathbb{P}(\phi(W_\ell) \text{ is sparse}) \lesssim \sum_{k=1}^{\delta n} \binom{n}{k} \binom{n-\ell}{k-1} \min(1-\alpha,q^{-1} + \frac{1}{\sqrt{\alpha k}})^{n-\ell-k+1}
\]
where $q$ is the smallest residue field for a maximal ideal of $R$. If we choose $\delta$ sufficiently small than this quantity is bounded by $O(e^{-c \alpha n})$.

\subsection{Unsaturated submodules} \label{sec:unsat}

The singularity bounds of Tao and Vu in \cite{TV06b}, \cite{TV07} are based on swapping the columns of the matrix $A$ with new columns drawn from a more singular probability distribution. Informally, if we have random vectors $X$ and $Y$ such that
\[
  \mathbb{P}(X \in V) \leq c \mathbb{P}(X \in V)
\]
for some $0 < c < 1$ and all $V$ in some class of vector spaces, then we can conclude that
\[
  \mathbb{P}(X_2, ..., X_n \text{ span } V) \leq c^{n-1} \mathbb{P}(Y_2, ..., Y_n \text{ span } V)
\]
modulo some difficulties with linear independence. If we would like to show that the columns of $A$ span such $V$ with small probability, it therefore suffices to use much worse bounds (such as the trivial bound) for the ``swapped in'' vectors $Y$.

In \cite{M10a} the author showed that this argument can also be used in the finite field setting as long as the vector spaces are not close to saturation for the random vectors. Explicitly, it was required that
\[
  \abs{\mathbb{P}(X \in V) - \abs{V^\perp}^{-1}} \geq D \abs{V^\perp}^{-1}
\]
for some $D$ (about 10).

In our more general setting, the actual swapping lemma is a straightforward generalization of the swapping lemma from \cite{M10a}.  We will prove the following swapping lemma in Section~\ref{sec:swap}.
\begin{lemma} \label{lem:swap}
There exists a probability distribution $\nu \in R$ with min-entropy $\alpha/8$ with the following property. Suppose $Y \in R^n$ is a random vector with iid entries taken from $\nu$. Then for every submodule $N \lhd R^n$ that is not sparse, we have the inequality
\[
  \abs{\mathbb{P}(X \in N) - \abs{N^\perp}^{-1}} \leq \left( \frac12 + o(1) \right) \abs{\mathbb{P}(Y \in N) - \abs{N^\perp}^{-1}}.
\]
\end{lemma}

We need a new notation for linear independence. In addition to avoiding linear dependencies, it is important that our vectors not introduce additional cokernel. We therefore say that $Y_1, ..., Y_r$ are \emph{a subbasis in $N$} if $Y_1, ..., Y_r \in N$ and $R^n / \langle Y_1, ..., Y_r \rangle \cong R^{n-r}$. Equivalently, the vectors $Y_1, \ldots, Y_r$ are linearly independent modulo every maximal ideal of $R$.

As in \cite{TV06b} and \cite{M10a}, we require a dyadic decomposition on the magnitude of $\mathbb{P}(X \in N)$. We recall that the \emph{combinatorial codimension} of $N$ is the unique fraction $d_\pm \in n^{-1} Z^+$ such that
\[
  (1 - \alpha)^{d_\pm/n + 1/n} \leq \mathbb{P}(X \in N) \leq (1 - \alpha)^{d_\pm/n}.
\]

Since we have the trivial bounds $2^{-n} \leq \mathbb{P}(X \in N) \leq 1$ we see that there are at most $O(\alpha n^2)$ possible combinatorial codimensions, so that it suffices to estimate
\[
  \mathbb{P}(\phi(W_{\ell+1}) \text{ unsaturated with } d_\pm(W) = d_\pm) \leq O(e^{-cn})
\]
for all $d_\pm$.

Let $V^\perp$ denote the isomorphism class of $\phi(W_{\ell+1})^\perp$. Let $N$ be an unsaturated submodules with $N^\perp \cong V^\perp$ of combinatorial codimension $d_\pm$. Let $Y_1, ..., Y_r$ be iid copies of $Y$ given by Lemma~\ref{lem:swap} and let $X_1', ..., X_s'$ be iid copies of $X$. Then we have
\[
  \mathbb{P}(\phi(W_{\ell+1}) = N) = \frac{\mathbb{P}(\phi(W_{\ell+1}) = N \wedge Y_1, ..., Y_r, X_1', ..., X_s' \text{ are a subbasis in } N)}{\mathbb{P}(Y_1, ..., Y_r, X_1, ..., X_s' \text{ are a subbasis in } N)}.
\]
Because we have assumed that $R^n / \langle Y_1, ..., Y_r, X_1', ..., X_s' \rangle \cong R^{n-r-s}$, we can replace $r+s$ of the column vectors $X_{\ell + 1}, ..., X_n$ with $Y_1, ..., Y_r, X_1', ..., X_s'$ with the following proposition.
\begin{proposition} \label{prop:dotheswap}
Suppose $\phi(W_{\ell+1}) = N$ and $Z_1, ..., Z_j$ are simply independent in $N$. Then there is a subset $\sigma \subseteq [\ell + 1, n]$ with $\abs{\sigma} = r$ such that
\[
  \phi(\langle \{X_k\}_{k \notin \sigma}, Z_1, ..., Z_j \rangle) = N.
\]
Furthermore, we can choose $\sigma$ so that if $k \in \sigma$, we have $\ker H_{k+1} \longrightarrow H_k \cong R$.
\end{proposition}
\begin{proof}
It suffices to prove this proposition for $r = 1$, since the linear independence of $Z_1, ..., Z_j$ prevents a subsequent choice of $X_k$ from colliding with previously chosen vectors.

From the definition of $\phi$, we have
\[
  t(Z - v) = a_{\ell + 1} X_{\ell + 1} + \cdots + a_n X_n
\]
for some $v \in I R^n$ and coefficients $a_k \in R$. We can find $k$ such that $(a_k) = (t)$ as ideals in $R$. In particular, we have $a_k = tu$ for some $u \in R^\times$, and we therefore write
\[
  tX_k = u^{-1} t(Z - v) - u^{-1} a_{\ell + 1} X_{\ell + 1} - \cdots - u^{-1} a_n X_n.
\]
Now, any $y \in N$ satisfies
\[
  t(y - v') = b_{\ell + 1} X_{\ell + 1} + \cdots + a_n X_n
\]
but this equation can be rewritten to replace $X_k$ with $Z$.
\end{proof}
Now it is convenient to abbreviate events.  We define the events
\begin{align*}
  D_n &:= Y_1, ..., Y_r, X_1', ..., X_s' \in N \\
  E_N &:= Y_1, ..., Y_r, X_1', ..., X_s' \text{ are a subbasis in } N \\
  F_{N,\sigma} &:= \phi(\langle \{X_k\}_{k \notin \sigma}, Y_1, ..., Y_r, X_1', ..., X_s' \rangle) = N
\end{align*}
By Proposition~\ref{prop:dotheswap}, we have
\[
  \mathbb{P}(\phi(W_{\ell + 1}) = N) \leq \sum_{\substack{\sigma \subseteq [\ell + 1,n] \\ \abs{\sigma} = r+s}} \frac{\mathbb{P}(\{X_k\}_{k \in \sigma} \in N)}{\mathbb{P}(E_N)} \mathbb{P}(F_{N,\sigma})
\]

First we consider the ratio
\[
\frac{\mathbb{P}(\{X_k\}_{k \in \sigma} \in N)}{\mathbb{P}(E_N)}
\]
We can expand the denominator with conditional expectation,
\begin{align*}
  \mathbb{P}(E_N) &= \mathbb{P}(D_N) \mathbb{P}(E_N \mid D_N) \\
  &= \mathbb{P}(Y \in N)^r \mathbb{P}(X \in N)^s \mathbb{P}(E_N \mid D_N)
\end{align*}
For the numerator we need the following independence lemma.
\begin{proposition}
    Let $Z_1, ..., Z_j$ be independent random vectors in $R^n$.  Then we have the bound
    \[
      \mathbb{P}(Z_1, \ldots, Z_j \in N \mid Z_1, \ldots, Z_j \text{ are a subbasis in } N) \leq \prod_{\ell = 1}^j \mathbb{P}(Z_\ell \in N).
    \]
\end{proposition}
\begin{proof}
The proof is the same as in \cite{M10a} except for minor notational differences. Expanding the left hand side with conditional expectation,
\[
  \prod_{j=1}^r \mathbb{P}(Z_j \in N \mid Z_1, \ldots, Z_{j-1} \in N \text{ and } Z_1, ..., Z_r \text{ are a subbasis in } N)
\]
Let $U_j$ denote the $R$ span of $Z_1, \ldots, Z_{j-1}$. We claim that
\[
  \frac{\mathbb{P}(Z \in N \setminus U_j)}{\mathbb{P}(Z \notin U_j)} \leq \mathbb{P}(Z \in N).
\]
In fact,
\begin{align*}
    \mathbb{P}(Z \in N \setminus U_j) &= \mathbb{P}(Z \in U_j) \mathbb{P}(Z \in N \setminus U_j) + \mathbb{P}(Z \notin U_j) \mathbb{P}(Z \in N \setminus U_j) \\
    &\leq \mathbb{P}(Z \in U_j) \mathbb{P}(Z \notin U_j) + \mathbb{P}(Z \in N \setminus U_j) \mathbb{P}(Z \notin U_j) \\
    &= \mathbb{P}(Z \in N) \mathbb{P}(Z \notin U_j). \qedhere
\end{align*}
\end{proof}
Since $\{X_k\}_{k \in \sigma}$ are simply independent in $N$ by construction, we therefore conclude that
\[
  \mathbb{P}(\{X_k\}_{k \in \sigma} \in N) \leq \mathbb{P}(X \in N)^{r+s}.
\]
We now apply Lemma~\ref{lem:swap}. Since $N$ is unsaturated,
\[
  \frac{\mathbb{P}(X \in N)}{\mathbb{P}(Y \in N)} \leq \left( \frac12 + \frac1D + o(1) \right).
\]
Collecting terms, we see that so far we have shown that
\[
  \mathbb{P}(\phi(W_{\ell + 1}) = N) \leq \sum_{\substack{\sigma \subseteq [\ell + 1,n] \\ \abs{\sigma} = r+s}} \left( \frac12 + \frac1D + o(1) \right)^r \frac{1}{\mathbb{P}(E_N \mid D_N)} \mathbb{P}(F_{N,\sigma})
\]

Next we control $\mathbb{P}(E_N \mid D_N)$ from below. We can expand this probability into the product
\[
  \prod_{j=1}^r \mathbb{P}(Y_1, ..., Y_j \text{ s.c.~in } N \mid Y_1, ..., Y_{j-1} \text{ s.c.~in } N \text{ and } Y_1, ..., Y_r \in N)
\]
along with analogous terms for $X'$. We can write each term of the product in the form
\[
  \prod_{\mathfrak{m}} (1 - \mathbb{P}(Y_j \in \langle Y_1, ..., Y_{j-1} \rangle + \mathfrak{m}R^n \mid Y_j \in N)).
\]
However, we can bound this with Odlyzko's lemma and the condition on the combinatorial codimension of $N$.

To finish the argument, we now sum over all submodules $N$ of prescribed isomorphism class. We then find that
\[
  \mathbb{P}(\phi(W_{\ell + 1}) \text{ is unsat}) \lesssim \sum_{\substack{\sigma \subset [\ell+1,n] \\ \abs{\sigma} = r+s}} (\frac12 + \frac1D + o(1))^r \sum_{N \text{ unsat}} \mathbb{P}(F_{N,\sigma}).
\]
We see that the inner sum on the right hand side is always bounded by $1$, since the collection of vectors can only induce a single submodule. We then choose $r$,$s$ appropriately to bound the whole quantity by $O(e^{-c \alpha n})$. \qed

\subsection{Swapping Lemma} \label{sec:swap}

We will now prove a generalization of the swapping lemma from \cite{TV06b}, \cite{TV07}, \cite{M10a}.
\begin{proof}[Proof of Lemma~\ref{lem:swap}]
We will abbreviate $\gamma = 1/8$ for convenience. Let
\[
  \nu(t) := \begin{cases} \gamma \mu * \mu^-(t), & t \neq 0 \\ 1 - \sum_{s \neq 0} \nu(s), & t = 0. \end{cases}
\]
Here $\mu^-(t) := \mu(-t)$. It is trivial to verify that $\nu$ is a probability distribution and $\widehat{\nu} > 1 - 2 \gamma$. It also has min-entropy $\beta = \gamma \alpha = \alpha/8$. 

The Fourier transform of $\nu$ is
\[
  \widehat{\nu} = 1 - \gamma + \gamma \abs{\widehat{\mu}}^2.
\]

It now suffices to verify the swapping inequality. We observe that
\[
  \mathbb{P}(X \in N) - \abs{N^\perp}^{-1} = \abs{N^\perp}^{-1} \sum_{\psi \in N^\perp \setminus \{0\}} \mathbb{E} \psi(X) = \abs{N^\perp}^{-1} \sum_{\psi \in N^\perp \setminus \{0\}} \prod_{\ell = 1}^n \widehat{\mu}(\psi_\ell)
\]
and similar for $\mathbb{P}(Y \in N)$. We therefore define
\begin{align*}
  f(\psi) &:= \prod_{\ell = 1}^n \abs{\widehat{\mu}(\psi_\ell)} \\
\intertext{and}
  g(\psi) &:= \prod_{\ell = 1}^n \widehat{\nu}(\psi_\ell)
\end{align*}
so that it suffices to show that
\[
  \sum_{\psi \in N^\perp \setminus \{0\}} f(\psi) \leq \left( \frac12 + o(1) \right) \sum_{\psi \in N^\perp \setminus \{0\}} g(\psi).
\]
We do this by level sets. For $u > 0$ define
\begin{align*}
  F(u) := \{ \psi \in N^\perp \mid f(\psi) \geq u \} \\
\intertext{and}
  G(u) := \{ \psi \in N^\perp \mid g(\psi) \geq u \}
\end{align*}
and likewise let $F'(u) = F(u) \setminus \{0\}$ and $G'(u) = G(u) \setminus \{0\}$.

Let $\eps > 0$ be chosen later. We must split the sum for $f$ into two parts: those frequencies $\psi$ where $f(\psi) \leq \eps$ and those where $f(\psi) > \eps$.

We first claim that $f(\psi) \leq g(\psi)^4$. This follows from the pointwise estimate $\abs{\widehat{\mu}(t)} \leq \widehat{\nu}(t)^4$ by the arithmetic-geometric mean inequality,
\[
  (\abs{\widehat{\mu}(t)}^2)^{1/8} \leq \frac18 (\abs{\widehat{\mu}}^2 + 7) = \nu(t).
\]
This controls those frequencies where $f$ is small. In fact,
\[
  \sum_{\substack{\psi \in N^\perp \setminus \{0\} \\ f(\psi) \leq \eps}} f(\psi) \leq \eps^{3/4} \sum_{\substack{\psi \in N^\perp \setminus \{0\} \\ f(\psi) \leq \eps}} g(\psi)
\]
so as long as $\eps \to 0$ as $n \to \infty$ this portion is done.

Now we will consider those frequencies where $f$ is large. In this case, we will apply Kneser's inequality to the sumset inclusion $F(u) + F(u) \subseteq G(u)$.

It suffices to show that $\abs{\widehat{\mu}(t)\widehat{\mu}(s)} \leq \widehat{\nu}(t + s)^2$. As in \cite{M10a} we consider two cases. If either $\widehat{\mu}(t) < 1 - 4 \gamma$ or $\widehat{\mu}(t) < 1 - 4 \gamma$ then the inequality is trivial from the lower bound for $\widehat{\nu}$. Otherwise, we write $\abs{\widehat{\mu}(t)} = 1 - \theta_1$ and $\abs{\widehat{\mu}(s)} = 1 - \theta_2$. By Lemma~\ref{lem:specadds} we have $\abs{\widehat{\mu}(t + s)} \geq 1 - 2(\theta_1 + \theta_2)$. But by the definition of $\nu$ we get $\widehat{\nu}(t + s) = 1 - \gamma + \gamma \abs{\widehat{\mu}(t + s)}^2 \geq \abs{\widehat{\mu}(t) \widehat{\mu}(s)}$.

Kneser's theorem tells us that $2 \abs{F(u)} \leq \abs{\Sym F(u) + F(u)} + \abs{G(u)}$. We would like to show that $\Sym (F(u) + F(u)) = \{0\}$. It suffices to show that $G(u)$ does not contain any non-trivial subgroup, as this would in turn guarantee that $F(u) + F(u)$ does not either.

Let $H \lhd G(u)$ be minimal, so that $H \cong \Z/p\Z$ for some prime $p$. Choose $w \in N^\perp$ that generates $H$ as a $\Z/p\Z$-module; since $N$ is not sparse, we can assume that $w$ contains at least $\delta n$ non-zero entries.

Define the function
\[
  h(t) := \sum_{\ell = 1}^n 1 - \widehat{\nu}(t_\ell)^2
\]
Averaging $h$ over $H$,
\[
  p^{-1} \sum_{t \in \Z/p\Z} h(t w) = n - p^{-1} \sum_{t \in \Z/p\Z} \sum_{\ell = 1}^n \widehat{\nu}(t w_\ell)^2
\]
By Plancherel's theorem and the fact that $N$ is not saturated we see that this entire quantity is bounded below by $\gamma \alpha \delta n$. Therefore we simply require that $u \leq \gamma \alpha \delta n$, which gives us the bound
\[
  \sum_{\substack{\psi \in N^\perp \setminus \{0\} \\ f(\psi) > \eps}} f(\psi) \leq \frac12 \sum_{\substack{\psi \in N^\perp \setminus \{0\} \\ f(\psi) > \eps}} g(\psi)
\]
as required.
\end{proof}

\subsection{Semi-saturated submodules} \label{sec:semisat}

Let $N$ be a semi-saturated submodule.  By the inverse Fourier transform,
\begin{align*}
    e^{-dn} &\leq \abss{\mathbb{P}(X \in N) - \frac{1}{\abs{N^\perp}}} \\
    &\leq \frac{1}{\abs{N^\perp}} \sum_{\xi \in N^\perp \setminus \{0\}} \prod_{\ell = 1}^n \abs{\widehat{\mu}(\xi_\ell)}
\end{align*}
In particular, we can find a $\zeta \in N^\perp \setminus \{0\}$ such that
\[
  \exp(-dn) \leq \exp(-\frac12 \sum_{\ell = 1}^n \psi(\zeta_\ell) ),
\]
with $\psi(t) := 1 - \abs{\widehat{\mu}(t)}^2$; taking logarithms gives
\[
  \sum_{\ell = 1}^n \psi(\zeta_\ell) \leq 2 dn.
\]
Let $\kappa$ be a parameter to be chosen later.  For all but $\kappa$ of $\xi_\ell$, we have
\[
  \psi(\zeta_\ell) \leq \frac{2 dn}{\kappa}
\]
or, equivalently, $\zeta_\ell \in \Spec_{1 - \eps} \mu$ for $\eps = \frac{dn}{\kappa}$.

\begin{proposition}
For all $\beta > 0$ there exists $\eps > 0$ such that for all orders $T$, $\Spec_{1 - \eps} \mu$ contains at most $\beta T$ elements of order $T$.
\end{proposition}
\begin{proof}
We first notice that there is an $\eta > 0$ such that $\Spec_{1 - \eta} \mu$ does not contain any non-trivial additive cosets $H + s$ with $H \lhd \widehat{R}$ and $s \in \widehat{R}$. In fact, for each such $H$ and $s$ we can apply Markov's inequality and use the fact that $\mu$ has min-entropy $\alpha$
\[
  (1 - \eta)^2 \#(H + s \cap \Spec_{1 - \eta} \mu) \leq \sum_{h \in H} \abs{\widehat{\mu}(h + s)}^2 \leq \abs{H} (1 - \alpha)
\]
We therefore choose $\eta = \alpha/2$. With this value, we apply the iterated form of Kneser's inequality to find
\[
  k \abs{\{t \in \widehat{R} \mid \abs{t} = T\} \cap \Spec_{1 - \eps} \mu \setminus \{0\}} \leq \abs{\{t \in \widehat{R} \mid \abs{t} = T\} \cap \Spec_{1 - k^2 \eps} \mu \setminus \{0\}}
\]
so we pick $k = \beta^{-1}$ and $\eps = k^{-2} \eta$.
\end{proof}

We now choose $\eps$, and therefore $d$, such that $\Spec_{1 - \eps} \mu$ has small cardinality in every finite torsion submodule of $\widehat{R}$; namely $\beta T$ for order $T$ elements.

We can count directly the number of vectors $\zeta$ of given order with the above constraint. In fact, for order $T$ we have
\[
  \#\{\text{structured } \zeta \text{ of order } T\} \leq \binom{n}{\kappa} T^\kappa (\beta T)^{n-\kappa} \lesssim \beta^n T^n
\]
where in the last inequality we chose $\kappa = n/10$ and replaced $\beta$ with a comparable value.

We can count the number of submodules $N$ with prescribed perpendicular isomorphism class and perpendicular to a fixed vector.
\begin{proposition}
Let $W^\perp$ be a fixed finite submodule of $\widehat{R}^n$ and let $\zeta \in \widehat{R}^n$ have finite order $T$. Then the number of submodules $N \lhd R^n$ with $N^\perp \cong W^\perp$ and $\zeta \in N^\perp$ is bounded by
\[
\#\{ N \mid \zeta \in N^\perp\} \lesssim \frac{\abs{W^\perp}^n}{\abs{\Aut W^\perp} T^n} \#\{\zeta \in W^\perp \mid \abs{\zeta} = T\}
\]
\end{proposition}
\begin{proof}
We double count the number of pairs $\zeta \in N^\perp$,
\[
  \sum_{\abs{\zeta} = T} \#\{N \mid \zeta \in N^\perp\} = \#\{(N^\perp, \zeta) \mid \zeta \in N^\perp, \abs{\zeta} = T\} = \sum_N \#\{\zeta \in N^\perp \mid \abs{\zeta} = T\}
\]
We observe that $\#\{N \mid \zeta \in N^\perp\}$ is independent of $\zeta$. In fact, for any two $\zeta_1, \zeta_2$ of order $T$ we can find an automorphism $\widehat{R}^n \to \widehat{R}^n$ that maps $\zeta_1$ to $\zeta_2$; this induces a correspondence between submodules of the same isomorphism class. We conclude that
\[
  \#\{N \mid \zeta \in N^\perp\} = \frac{\#\{N \mid N^\perp \cong W^\perp \} \#\{\zeta \in N^\perp \mid \abs{\zeta} = T\}}{\#\{\zeta \in \widehat{R}^n \mid \abs{\zeta} = T\}}.
\]
We have the trivial bound
\[
  \#\{N \mid N^\perp \cong W^\perp\} \leq \frac{\abs{N^\perp}^n}{\abs{\Aut N^\perp}},
\]
so collecting terms we find
\[
  \#\{ N \mid \zeta \in N^\perp\} \lesssim \frac{\abs{W^\perp}^n}{\abs{\Aut W^\perp} T^n} \#\{\zeta \in W^\perp \mid \abs{\zeta} = T\} \qedhere
\]
\end{proof}

Now we are ready to estimate the probability that $W$ is semi-saturated. We count the number of semi-saturated spaces,
\begin{align*}
  \mathbb{P}(W \text{ is semi-saturated}) &\leq \sum_{N \text{ semi-saturated}} \mathbb{P}(W = N) \\
  &\leq \#\{N \text{ semi-saturated} \mid N^\perp \cong W^\perp \} \left( \frac{D}{\abs{W^\perp}} \right)^{n-\ell}
\end{align*}
With the above bounds on the number of structured vectors and the number of submodules perpendicular to a given vector, we bound
\[
  \#\{N \mid \text{semi-sat and } \zeta \in N^\perp \text{ struct., order } T\} \lesssim \beta^n \frac{\abs{W^\perp}^n}{\abs{\Aut W^\perp}} \#\{\zeta \in W^\perp \mid \abs{\zeta} = T\}
\]
Summing over every possible order $T$,
\[
  \#\{N \mid \text{semi-sat and } \zeta \in N^\perp \text{ struct.}\} \lesssim \beta^n \frac{\abs{W^\perp}^n}{\abs{\Aut W^\perp}} \abs{W^\perp}.
\]
We now combine this with the bound on $\mathbb{P}(W = N)$ to find
\[
  \mathbb{P}(W \text{ semi-sat}) \lesssim \beta^n D^{n-\ell} \frac{\abs{W^\perp}^\ell}{\abs{\Aut W^\perp}} \abs{W^\perp}.
\]
By construction of $W$, its perpendicular module $W^\perp$ must contain at least $\ell$ terms of maximal order, so
\[
  \abs{\Aut W^\perp} \geq \abs{W^\perp}^\ell.
\]
We also know that $\abs{W^\perp} \leq D e^{cn}$ since $W$ is semi-saturated, so if we take $\beta$ (and therefore $d$) sufficiently small then we find
\[
  \mathbb{P}(W \text{ is semi-saturated}) = O(e^{-cn})
\]
as required. \qed

\section{Acknowledgements}

The author would like to thank Terence Tao for helpful criticism and guidance.

\bibliography{cokernel.submit.v1}{}
\bibliographystyle{abbrv}

\end{document}